\documentclass[12pt]{amsart}
\usepackage{all2021,xypic}
\usepackage{color}
\usepackage{enumitem}

\usepackage{geometry}

\newcommand{\Sh}{\mathrm{Sh}}
\renewcommand{\phi}{\varphi}
\renewcommand{\Vec}{\mathbf{Vec}}

\DeclareMathOperator{\Span}{Span}
\DeclareMathOperator{\II}{\mathcal{I}}
\DeclareMathOperator{\subrk}{subrk}

\let\olditemize\itemize
\renewcommand{\itemize}{
  \olditemize
  \setlength{\itemsep}{1em} 
}

\numberwithin{equation}{section}

\begin{document}
\title{A Functorial Version of Chevalley's Theorem on Constructible Sets}
\author{Andreas Blatter} 
\thanks{AB was supported by Swiss National Science
Foundation grant 200021\_191981}
\email{andreas.blatter@unibe.ch}
\address{Mathematical Institute, University of Bern,
Alpeneggstrasse 22, 3012 Bern, Switzerland}
\maketitle

\begin{abstract}
    To determine whether an $n\times n$-matrix has rank at most $r$ it suffices to check that the $(r+1)\times (r+1)$-minors have rank at most $r$. In other words, to describe the set of $n\times n$-matrices with the property of having rank at most $r$, we only need the description of the corresponding subset of $(r+1)\times (r+1)$-matrices. We will generalize this observation to a large class of subsets of tensor spaces. A description of certain subsets of a high-dimensional tensor space can always be pulled back from a description of the corresponding subset in a fixed lower-dimensional tensor space.
\end{abstract}

\section{Introduction}




\subsection{Polynomials of Tensors}

Let $K$ be either $\RR$ or $\CC$, $A \in (K^n)^{\otimes d_1}$ and $B \in (K^n)^{\otimes d_2}$. There is a natural way to multiply $A$ and $B$ by taking the tensor product: Writing $A=(a_{i_1 \hdots i_{d_1}})_{i_1 \hdots i_{d_1}}$, $B=(b_{j_1 \hdots j_{d_2}})_{j_1 \hdots j_{d_2}}$, then

$$A \otimes B = (a_{i_1 \hdots i_{d_1}}b_{j_1 \hdots j_{d_2}})_{i_1 \hdots i_{d_1}j_1 \hdots j_{d_2}} \in (K^n)^{\otimes d_1+d_2}.$$

If $d_1=d_2$, we can also add $A$ and $B$ simply by component-wise addition $A+B$. And we can also perform scalar multiplication $A\mapsto \lambda A$ for some $\lambda \in K$. Combining these operations we obtain polynomials of tensors that map from some direct sum of tensor spaces $(K^n)^{\otimes d_1} \oplus \hdots \oplus (K^n)^{\otimes d_m}$ to some tensor space $(K^n)^{\otimes e}$.

\begin{ex} \label{example:tensor}
    Let $d$ and $r$ be fixed,  
    \begin{align*}
        \alpha_n: ((K^n)^{\otimes 1})^{\oplus d\cdot r} &\to (K^n)^{\otimes d}\\
        (v_{11}, \hdots, v_{1d}, \hdots, v_{r1}, \hdots, v_{rd}) &\mapsto \sum_{i=1}^r v_{i1}\otimes\hdots\otimes v_{id}.
    \end{align*}

    The image of $\alpha_n$ is the set of tensors in $(K^n)^{\otimes d}$ with tensor rank smaller or equal to $r$.
\end{ex}

\begin{ex}
    Let
    \begin{align*}
    \alpha_n: ((K^n)^{\otimes 2})^{\oplus 3} &\to (K^n)^{\otimes 4}\\
    (A, B, C) &\mapsto A \otimes B - C\otimes C
    \end{align*}
\end{ex}

\subsection{Images of Tensor Polynomials} We now ask how we can describe the image of such a polynomial $\alpha_n$. If we fix the integer $n$, then there already exists a satisfying answer: In case $K=\CC$, Chevalley's Theorem on constructible sets says that an image of a constructible set under a polynomial map is again constructible. A constructible set in some complex vector space, say $V$, is a set that can be described by a finite boolean combination of polynomial equations and inequations, or in other words, it is a finite union of sets of the form

$$\{v \in V: f_1(v)=\hdots=f_k(v)=0 \text{ and } g(v) \neq 0\},\,\,\,f_1, \hdots, f_k, g \in \CC[V].$$

So, in conclusion, by Chevalley's Theorem, the image of $\alpha_n$ is constructible, since it is the image of a whole (constructible) vector space, under a polynomial map.\\

If $K=\RR$ there is an analogous theorem by Tarski-Seidenberg that says that an image of a semialgebraic set under a polynomial map is again semialgebraic. A semialgebraic set is also a set that can be described by finitely many equations and inequations, but when we use the word inequation in this setting, we do not just mean "$\neq$", but also "$>$" and "$\geq$". So, also when working over $\RR$, the image of $\alpha_n$ can be described by finitely many equations and inequations.



\subsection{Images of Infinite Collections of Tensor Polynomials}

The point of this paper is that we do not want $n$ to be fixed, but we want to give a finite (implicit) description of the whole collection of the images, say $(\im(\alpha_n))_n.$ Our main result basically says that there exists $m\in \NN$ such that $\im(\alpha_m)$ already completely describes the whole collection $(\im(\alpha_n))_n.$ To make this more precise, note that if $\phi: K^n \to K^m$ is a linear map, we get a linear map

\begin{align*}
    \phi^{\otimes e}:(K^n)^{\otimes e} &\to (K^m)^{\otimes e}\\
    v_1 \otimes \hdots \otimes v_e &\mapsto \phi(v_1) \otimes \hdots \otimes \phi(v_e)
\end{align*}

Note that $(\im(\alpha_n))_n$ is invariant under $\phi^{\otimes e}$, meaning that if $\phi: K^n \to K^m$, $p\in \im(\alpha_n)$, then $\phi^{\otimes e}(p) \in \im(\alpha_m)$. We claim, that there exists $m\in \NN$ such that for every $n\in\NN$

$$\im(\alpha_n) = \{p \in (K^n)^{\otimes e}:\text{for every linear map }\phi:K^n \to K^m,\,\, \phi^{\otimes e}(p) \in \im(\alpha_m)\}.$$

So, the (in-)equations that describe $\im(\alpha_n)$ can be pulled back from the (in-)equations that describe $\alpha_m$. In the case of Example \ref{example:tensor} it has been known before that this works with $m=d$, but this is an easier example, because it considers only a polynomial on vectors and not higher dimensional tensors.

\subsection{The Results} Our results are slightly more general than the situation described above. Let us use coordinate-free notation from now on. Note that $V \mapsto (V)^{\otimes d_1} \oplus \hdots \oplus (V)^{\otimes d_m}$ is a functor from the category of finite-dimensional $K$-vector-spaces to itself. In general, a polynomial functor, $P$, is a functor of this form, or a subfunctor thereof (e.g. $S^2$, which is a subfunctor of $V\mapsto V^{\otimes 2}$, is also a polynomial functor, see section $\ref{section:polfunc}$ for a precise definition).\\

The functorial equivalents of polynomial maps are called polynomial transformations (see section \ref{section:poltrans}), but they are essentially just combinations of tensoring and component-wise addition, as above. Finally, a constructible/semialgebraic subset $X \subseteq P$ assigns to every vector space $V$ a constructible/semialgebraic set $X(V)\subseteq P(V)$, such that for every $\phi \in \Hom(V, W)$, $P(\phi)(X(V)) \subseteq X(W)$, and that is determined by a specific vector space $U$ (which plays the same role as the integer $m$ in the previous section), see also Definition \ref{definition:constructible}.\\

With this language we will present a proof of a functorial version of Chevalley's Theorem: The image of any constructible subset under a polynomial transformation is again a constructible subset (Theorem \ref{Theorem:chevalley}). In the real case, we do not have a perfect equivalent for Tarski-Seidenberg's Theorem, i.e. we do not know if the image of any semialgebraic subset is again semialgebraic, but we can prove that the image of any closed subset (i.e. $X(V)$ is closed for every vector space $V$) is semialgebraic.

\subsection{Structure of the Paper}

In section \ref{section:preliminaries}, we will give a complete introduction to polynomial functors. In section \ref{section:constsemi}, we define the main objects of this paper, constructible/semialgebraic subsets of polynomial functors. Section \ref{section:parameterisation} gives a parameterisation result for constructible subsets, i.e. we will show that every constructible subset is a union of images of some ``nice" (in particular, closed) subsets. This result is obviously wrong for semialgebraic subsets. So, in order to prove our version of Chevalley's Theorem in section \ref{section:chevalley}, we will only have to prove that images of these nice subsets are constructible, which is essentially all we can prove in the real case.

\subsection{Related Work}

The main result in \cite{Draisma17} implies that the (Zariski-)closure of every image of a polynomial transformation is constructible/semialgebraic in the sense above. In \cite{Blatter22}, the corresponding result over finite fields is proven. In \cite{Bik21}, a similar-looking version of Chevalley's Theorem is proven: Instead of working with an infinite collection of finite-dimensional spaces, they work in one infinite-dimensional space, namely the projective limit of $P(K^1), P(K^2), P(K^3), \hdots$. A constructible set in such a space is a subset that is given by finitely many equations and inequations, and that is invariant under an action of an inductive limit of $(\GL_n)_n$. They prove that the image of a constructible set in this sense is again constructible. However, this result seems to be essentially different from ours, since attempts to derive our result from this have failed.

\subsection*{Acknowledgements}

I thank Jan Draisma for the helpful discussions, helping me find some of the proofs and proof-reading the paper.

\section{Preliminaries} \label{section:preliminaries}

We denote by $\CC$ the field of complex numbers, although this can be
replaced by any other algebraically closed field of characteristic 0.
Similarly, $\RR$ denotes the field of real numbers, or any real closed
field of characteristic 0 (i.e. a field that is not algebraically closed,
but becomes algebraically closed when adjoining the square root of $-1$).
The letter $K$ may refer to both $\CC$ and $\RR$.

\subsection{Polynomial Functors}\label{section:polfunc}

Let $\Vec$ be the category of finite-dimensional $K$-vector spaces. We
write $\Hom(U,V)$ for the space of $K$-linear maps $U \to V$.

\begin{de} \label{definition:polfunctor}
A {\em polynomial functor} over $K$ is a covariant
functor $P:\Vec \to \Vec$ such that for any $U,V \in \Vec$ the
map $P:\Hom (U,V) \to \Hom(P(U),P(V))$ is polynomial of degree at
most some integer $d$ that does not depend on $U$ or $V$.
\end{de}

The phrase ``the map $P:\Hom (U,V) \to \Hom(P(U),P(V))$ is polynomial"
means that when choosing bases for $U$, $V$, $P(U)$ and $P(V)$, the map P that maps the 
matrix representation of $\phi \in \Hom (U,V)$ to the matrix representation
of $P(\phi) \in \Hom(P(U),P(V))$ must be polynomial. This notion is 
independent of the choice of bases.\\

\begin{ex}\,
    \begin{enumerate}
        \item For a fixed $U \in \Vec$, the {\em constant functor} $P:V \mapsto U$, $\phi \mapsto \id$.
        \item The {\em identity functor} $T:V\mapsto V$, $\phi \mapsto \phi$.
        \item The {\em $d$-th direct sum} $T^{\oplus d}:V \mapsto V^{\oplus d}$, $\phi \mapsto \phi^{\oplus d}$.
        \item The {\em $d$-th tensor power} $T^{\otimes d}:V \mapsto V^{\otimes d}$, $\phi \mapsto \phi^{\otimes d}$. 
    \end{enumerate}
\end{ex}

The following definitions will allow us to give an intuitive
characterisation of all polynomial functors.

\begin{de}
    Let $P:\Vec \to \Vec$ be any functor. A functor $Q:\Vec \to \Vec$ is
    called a {\em subfunctor} of $P$ if $Q(V) \subseteq P(V)$
    for all $V$ and $Q(\varphi) = P(\varphi)|_{Q(V)}$ for all $\varphi \in
    \Hom(V, W)$
\end{de}

\begin{de}
    Let $P, Q:\Vec \to \Vec$ be functors. Notions like $P \oplus Q$, $P \otimes Q$ and, in case $Q$ is a subfunctor of $P$, $P/Q$ are defined elementwise in the obvious way (e.g. $(P \oplus Q)(V) := P(V) \oplus Q(V)$, and similarly for morphisms).
\end{de}

Now for the characterisation of polynomial functors, see e.g. \cite{Bik2020} for more detailed explanations:

\begin{prop} \label{proposition:polfuncchar}
    Let $P: \Vec \to \Vec$ be a functor. The following are equivalent:

    \begin{enumerate}
        \item $P$ is a polynomial functor.
        \item $P$ is isomorphic to a finite direct sum of subfunctors of $T^{\otimes d}$ and a constant polynomial functor.
        \item $P$ is isomorphic to a finite direct sum of quotients of $T^{\otimes d}$ and a constant polynomial functor.
    \end{enumerate}
\end{prop}

In other words, the set of polynomial functors is the smallest set of functors that contains constant functors and the identity functor, and is
closed under taking direct sums, tensor products and subfunctors (or
quotients). The irreducible polynomial functors (i.e. the polynomial
functors that cannot be written as a nontrivial direct sum) are exactly the
Schur functors (see e.g. \cite{Fulton91} for a definition of Schur functors) and the constant functor $V\mapsto K^1$. Every
polynomial functor has a unique decomposition as a direct sum of a constant
functor and Schur functors (e.g. $T^{\otimes 2} = S^2 \oplus \bigwedge^2$)

\begin{re}
    The requirement from Definition \ref{definition:polfunctor} that the
    degree of the maps $P(\phi)$ must be universally bounded rules out
    examples like $V \mapsto \bigwedge^0(V) \oplus \bigwedge^1(V) \oplus 
    \bigwedge^2(V) \oplus \hdots$
\end{re}

\subsection{Gradings}
By Proposition \ref{proposition:polfuncchar} we can write a polynomial functor $P$ as $P = P_0  \oplus P_1 \oplus \hdots \oplus P_d$, where $P_0$ is a constant polynomial functor, and for $e\geq 1$, $P_e$ is a subfunctor (or quotient) of some $(T^{\otimes e})^{\oplus m_e}$. This decomposition is unique (up to adding zero-spaces), and $P_e$ is called the {\em degree-$e$-part} of $P$. We can also define it without using the characterization by

$$P_e(V):=\{p \in P(V)\mid \text{ for all } t \in K, P(t\cdot \id_V)p = t^e\cdot p\}.$$

\noindent $P_0$ is also called the {\em constant part} of $P$. We call $P$ {\em pure} if $P_0$ is the zero-space, and we call $P_1 \oplus \hdots \oplus P_d$ the {\em pure part} of $P$. This is also notated as $P_{\geq 1}$. Terms like $P_{\leq e}$ or $P_{>e}$ are defined accordingly in the obvious way.

\subsection{An Order on Polynomial Functors} \label{subsection:order}

\begin{de}
We call a polynomial functor $Q$ smaller than a polynomial
functor $P$, if the two are not isomorphic, and for the largest 
$e$ such that $Q_e$ is not isomorphic to $P_e$, $Q_e$ is 
isomorphic to a quotient of $P_e$.
\end{de}

Writing these largest nonisomorphic parts $Q_e$ and $P_e$ as sums of Schur functors, i.e. 

$$Q_e = \bigoplus_{\lambda:|\lambda|=e} (S^{\lambda})^{m_{\lambda}},\,\,\, P_e = \bigoplus_{\lambda:|\lambda|=e} (S^{\lambda})^{n_{\lambda}}$$

then $Q$ is smaller than $P$ if and only if $m_{\lambda} \leq n_{\lambda}$ for all partitions $\lambda$ of $e$ (where the inequality is strict for at least one such $\lambda$). This also demonstrates that this order on polynomial functors is a well-founded order (i.e. there are no infinite strictly decreasing chains).

\subsection{Subsets}

\begin{de} \label{definition:subset}
    Let $P$ be a polynomial functor over $K$. A {\em subset} of $P$, $X \subseteq P$, consists of a subset $X(V) \subseteq P(V)$ for each $V \in \Vec$, such that for all $\phi \in \Hom(V, W)$ and $v \in X(V)$ we have $P(\phi)(v) \in X(W)$.
\end{de}

\begin{ex} \label{example:tensor2}
    Let $P=T^{\otimes d}$, $r\in\NN$ fixed. Then, $X\subseteq P$ given by

    $$X(V) = \{A \in P(V): \rk{A} \leq r\}$$

    is a subset (compare Example \ref{example:tensor}). 
\end{ex}

\begin{ex}
    Let $P$ be any polynomial functor, and $A$ any subset of $P_0$. Then

    $$X(V):=\{(a, b) \in P(V)=P_0(V)\oplus P_{\geq 1}(V)\mid a\in A\}$$

    is a subset, usually denoted by $A \times P_{\geq 1}$. We use the notation with $\times$ instead of $\oplus$ purely for aesthetic reasons. We will often consider sets of the form $A \times Q$, where $A$ is an affine variety, and $Q$ is a pure polynomial functor. These can be implicitly seen as subsets of $K^n \oplus Q$, where $n$ is big enough such that there is an embedding of $A$ into $K^n$.
\end{ex}

\begin{de}
    A subset $X\subseteq P$ is called {\em closed}, if $X(V)$ is closed (i.e. the zero-locus of a finite collection of polynomials) for every $V\in \Vec$. $X$ is called reducible, if there exist closed subsets $X_1, X_2 \subsetneq X$ such that $X = X_1 \cup X_2$, and irreducible if it is not reducible.
\end{de}

For closed subsets of polynomial functors there exists an important Noetherianity result by Draisma:

\begin{thm}[\cite{Draisma17}] \label{Theorem:Noetherianity}

Any descending chain of closed subsets of a polynomial functor

$$P \supseteq X_1 \supseteq X_2 \supseteq X_3 \supseteq \hdots$$

stabilizes, i.e. there exists $N \in \NN$ such that $X_N=X_{N+1}=X_{N+2}=\hdots$.

\end{thm} 

This theorem implies in particular that a closed subset has a finite number of irreducible components, i.e. inclusion-wise maximal irreducible subsets.

\subsection{Polynomial Transformations} \label{section:poltrans}

We now define the functorial equivalent of a polynomial map:

\begin{de} \label{def:poltrans}
    A polynomial transformation $\alpha:Q \to P$ consists of a polynomial map $\alpha_V:Q(V) \to P(V)$ for each $V \in \Vec$, such that for all $\phi \in \Hom(V, W)$ the following diagram commutes: 
$$
    \xymatrix{
  Q(V) \ar[r]^{\alpha_V} \ar[d]^{Q(\phi)} & P(V) \ar[d]^{P(\phi)} \\
  Q(W) \ar[r]^{\alpha_W} & P(W) 
}$$
\end{de}

We will often consider polynomial transformations from sets of the form $A \times Q$, where $A$ is an affine variety and $Q$ is pure. These can simply be interpreted as restrictions of polynomial transformations as defined above.\\

Note that the image $X(V):=\im(\alpha_V)$ of any polynomial transformation is a subset.\\

\begin{ex}
    The tensor polynomials as described in the introduction are polynomial transformations, e.g. rewriting Example \ref{example:tensor} in the language of polynomial functors: $Q=T^{\oplus r\cdot d}$, $P=T^{\otimes d}$ and $\alpha$ given by
    \begin{align*}
        \alpha_V: Q(V) &\to P(V)\\
        (v_{11}, \hdots, v_{1d}, \hdots, v_{r1}, \hdots, v_{rd}) &\mapsto \sum_{i=1}^r v_{i1}\otimes\hdots\otimes v_{id}
    \end{align*}

    is a polynomial transformation, and its image is the subset from Example \ref{example:tensor2}.
\end{ex}

We make a few observations on the structure of polynomial transformations:

\begin{re}\label{remark:homtrans}
    Note that the diagram in Definition \ref{def:poltrans} in particular commutes if $\phi$ is a multiple of the identity, i.e. $ \phi = t\cdot \id$ with $t\in K$. Say $Q=Q_e$ is a homogeneous polynomial functor of degree $e$, $P=P_d$ is a homogeneous polynomial functor of degree $d$, and $\alpha:Q \to P$ is a polynomial transformation. Then, for $q \in Q(V)$:

    \[\alpha_V(Q(t\cdot \id_V)q) = \alpha_V(t^{e}q)\]

    is equal to 

    \[P(t\cdot \id_V)\alpha_V(q) = t^{d} \alpha_V(q).\]

    So, unless $\alpha$ is the zero-transformation, $e$ must divide $d$, and $\alpha_V$ is a homogeneous polynomial of degree $d/e$, if $e\neq 0$ (this needs $K$ to be an infinite field). In particular, if $d=e\neq 0$, then $\alpha$ is linear. Note that the only linear transformations from 

    \[Q = \bigoplus_{\lambda:|\lambda| = e} (S^{\lambda})^{\oplus m_{\lambda}}\,\, \text{     to     }\,\, P = \bigoplus_{\lambda:|\lambda| = e = d} (S^{\lambda})^{\oplus n_{\lambda}}\]

    are of the form 

    \[\alpha_V((q_{\lambda i})_{|\lambda|= e, 1 \leq i\leq m_{\lambda}}) = (p_{\lambda j})_{|\lambda| = d=e, 1 \leq j\leq n_{\lambda}}\]

    where 

    \[p_{\lambda j} = \sum_{i=1}^{m_{\lambda}} A_{\lambda i j} q_{\lambda i}, \, \, A_{\lambda i j} \in K.\]
    
    If $e=0$, we get that $\alpha_V(q) = t^d \alpha_V(q)$ so $d$ has to be equal to 0, but $\alpha$ need not be linear.
\end{re}

\begin{re} \label{remark: split2}
    Now let $P=P_d$ still be homogeneous, but $\alpha:B\times Q \to P$, where $Q$ is any pure polynomial functor and $B$ is an affine variety. Write $Q=Q_{<d} \oplus Q_d \oplus Q_{>d}$. Then, by a similar argument as above, we can write $\alpha$ as

    \[\alpha_V(b, q_{<d}, q_d, q_{>d}) = \alpha_{1, V}(b, q_{<d}) + \alpha_{2, V}(b, q_d)\]

    where $\alpha_2$ is of the same form as the linear transformation in the previous remark, except that the coefficients $A_{\lambda i j}$ are of the form $f_{\lambda i j}(b)$, where $f_{\lambda i j} \in K[B]$.
\end{re}

\subsection{Shifting}

\begin{de}
    Any fixed $U \in \Vec$ defines a polynomial functor $Sh_U$: $V \mapsto U \oplus V$, $\phi \mapsto \id_U \oplus \phi$. If $P$ is a polynomial functor, then $Sh_U P:= P \circ Sh_U$ is also a polynomial functor, called the {\em shift over $U$} of $P$. We also write $Sh_U X := X \circ Sh_U$ for subsets $X \subseteq P$, and, for polynomial transformations $\alpha:Q \to P$, $Sh_U \alpha := \alpha_{U\oplus V} : Sh_U Q \to Sh_U P$.
\end{de}

The concept of shifting is useful due to the following theorem: 

\begin{thm}[\cite{Bik21}, Theorem 5.1.] \label{Theorem:shift}
    Let $X \subseteq P$ a closed subset that is not of the form $\widetilde{X} \times P_d$ (where $P_d$ is the highest-degree part of $P$). Then there exist a vector space $U$ and a nonzero polynomial $h \in K[P(U)]$, such that 
    $$Sh_U(X)[1/h] = \{p \in X(U\oplus V):h(p)\neq 0\}$$

    (where $h$ is regarded as a polynomial on $P(U \oplus V)$ via the map $P(\pi_U):P(U \oplus V) \to P(U)$, where $\pi_U$ is the standard projection), is isomorphic to $B\times R$, where $B$ is an affine variety, and $R$ is a pure polynomial functor with $R< P_{\geq 1}$.
\end{thm}

This theorem will allow us to use induction on the order of polynomial functors, by identifying big subsets with subsets in smaller polynomial functors.

\section{Constructible and Semialgebraic Subsets of Polynomial Functors} \label{section:constsemi}

\subsection{Definition}

We can now introduce the main objects of this paper:

\begin{de} \label{definition:constructible}
    Let $P$ be a polynomial functor over $\CC$. A subset $X \subseteq P$ is called
    \begin{enumerate}
        \item {\em pre-constructible}, if $X(V)$ is constructible for every $V \in \Vec$
        \item {\em constructible}, if it is pre-constructible, and there exists $U \in \Vec$, such that for all $V \in \Vec$:
        \begin{align} \label{equation:pullback}
            X(V) = \{v \in P(V)| \forall \phi \in \Hom(V, U), P(\phi)(v) \in X(U)\}
        \end{align}
        We say that $X$ is determined by $U$.
    \end{enumerate}

    Replacing $\CC$ by $\RR$ and the word ``constructible" by the word ``semialgebraic" yields a definition for a (pre-)semialgebraic subset.
\end{de}

\begin{re}
    It is straightforward to check that if $X$ is determined by $U$, then it is also determined by any other vector space of dimension at least $\dim(U)$, in particular also by $K^n$ for $n\geq \dim(U)$.
\end{re}

Equation \eqref{equation:pullback} is a finiteness condition that makes
sure that all the information of $X(V)$, even if $V$ is very big, is 
already contained in $X(U)$. Note that it is natural to ask for a finiteness
condition when using the word ``constructible" (or ``semialgebraic"), since
also the classical notion of a constructible set refers to a finite union 
of locally closed sets.\\

Also note that the inclusion ``$\subseteq$" of equation 
$\eqref{equation:pullback}$ is true for all subsets. Hence, in order to 
check whether a pre-constructible subset $X\subseteq P$ is actually
constructible, it suffices to show that for all $v \in P(V) \setminus X(V)$
 there exists $\phi \in \Hom(V, U)$, such that $P(\phi)(v) \notin X(U)$.

\subsection{Examples}

In the following, we give some examples of constructible subsets over $\CC$. They are also semialgebraic subsets, if you replace the ground field by $\RR$.

\begin{ex}  \label{example:closed}
    If $X$ is a closed subset of $P$, i.e. it is a subset and $X(V)$ is Zariski-closed for every $V$, then Theorem \ref{Theorem:Noetherianity} implies that $X$ is a constructible subset. For example:
    \begin{enumerate}
        \item $P = T^{\otimes 2}$ and $X(V) = \{A \in P(V): \rk(A) \leq r\}$, i.e. matrices of rank at most some integer $r$.
        \item $P = T^{\otimes d}$ and $X(V) = \{A \in P(V): \text{slicerank}(A) \leq r\}$, i.e. tensors of slice rank at most some integer $r$ (see \cite{Tao16}).
        \item $P = T^{\otimes 3}$ and $X(V) = \{A \in P(V): \text{geometric rank}(A) \leq r\}$, i.e. tensors of geometric rank at most some integer $r$ (see \cite[Lemma 5.3.]{Kopparty20}).
    \end{enumerate}
\end{ex}

\begin{ex} \label{example:span}
    Let $P=T^{\oplus d+1}$ (where $d$ is fixed) and
    $$X(V) = \{(v_0, v_1, \hdots, v_d) \in P(V) : v_0 \in \Span(v_1, \hdots, v_d)\}$$
    This is a constructible subset determined by $\CC^1$: Let $(v_0, \hdots, v_d) \in P(V)\setminus X(V)$, i.e. $v_0 \notin \Span(v_1, \hdots, v_d)$. Then we can find a linear map $\phi:V \to \CC^1$, such that $v_1, \hdots, v_d$ are in the kernel of $\phi$, but not $v_0$.
\end{ex}

\begin{ex} \label{example:matroid}
    Let $P = T^{\oplus d}$ and $([d]=\{1, 2, \hdots, d\}, \II)$ be a matroid (see e.g. \cite{Oxley06}). Let
    $$\widetilde{X}_{\II}(V):=\{(v_1, \hdots, v_d) \in P(V) : \forall I \in 2^{[d]}, (v_j)_{j \in I} \text{ linearly independent} \Leftrightarrow I \in \II\}$$
    $$X_{\II}(V) := \bigcup_{g \in \End(V)} P(g)(\widetilde{X}(V)).$$

    An interesting example is $d=3$, $\II = \{I \in 2^{[d]}: |I| \leq 2\}$. It turns out that

    $$X_{\II} = \widetilde{X}_{\II} \cup \widetilde{X}_{\{\{1\}, \{2\}, \{3\}, \emptyset\}}\cup \widetilde{X}_{\{\{1\}, \{2\},  \emptyset\}}\cup \widetilde{X}_{\{\{1\},  \{3\}, \emptyset\}}\cup \widetilde{X}_{\{ \{2\}, \{3\}, \emptyset\}}\cup \widetilde{X}_{\{ \emptyset\}}$$

    (in particular, it does not include the sets $\widetilde{X}_{\{\{1\}, \emptyset\}}$ or $\widetilde{X}_{\{\{1, 2\},\{1\}, \{2\},  \emptyset\}}$). \\

    Each such $X_{\II}$ is a constructible subset determined by $\CC^{d}$, since for $(v_1, \hdots, v_d) \in P(V)\setminus X_{\II}(V)$, there exists a linear map $\phi:V\to \CC^d$, such that $\phi|_{\Span(v_1, \hdots, v_d)}$ is injective, so all linear independencies (and, trivially, all linear dependencies) in $(v_1, \hdots, v_d)$ are preserved, and hence $P(\phi)(v_1, \hdots, v_d) \in P(\CC^d)\setminus X_{\II}(\CC^d)$.
\end{ex}

The following Example also makes sense when replacing the number 3 by any other positive integer $d$, but we use the number 3 for ease of notation.
It is also an illustration of how our theory of single-variable polynomial functors could be generalized to multivariable polynomial functors, which we allow multiple linear maps to act on.


\begin{ex} \label{example:subrank}
     Let $P = T^{\otimes 3}$, $q \in \NN$ fixed, and 
    
    $$X(V):=\{A \in P(V): \subrk(A) \leq q\}$$
    
    where the subrank of $A$, $\subrk(A)$, is the biggest integer $q$, such that there exist linear maps $\phi_1, \phi_2, \phi_3: V \to \CC^q$ with 

    $$(\phi_1 \otimes \phi_2 \otimes \phi_3)A = e_1^{\otimes 3} + \hdots + e_q^{\otimes 3}.$$

\begin{itemize}
    
    \item We claim that $X$ is a constructible subset of $P$. It is clear
    that $X$ is a subset, and by quantifier elimination that every $X(V)$
    is constructible, so $X$ is pre-constructible.

    \item We claim that $X$ is determined by $\CC^{3(q+1)}$: Let $A \in P(V) \setminus X(V)$. Then there exist $\phi_1, \phi_2, \phi_3: V \to \CC^{q+1}$ with $(\phi_1 \otimes \phi_2 \otimes \phi_3)A = e_1^{\otimes 3} + \hdots + e_{q+1}^{\otimes 3}.$ 
    
    \item Let

    $$\Phi:=\phi_1 \oplus \phi_2 \oplus \phi_3 : V \to \CC^{3(q+1)}$$

    \noindent Then $P(\Phi)A$ has subrank at least $q+1$, i.e. it does not lie in $X(\CC^{3(q+1)})$, because

    $$(\pi_1 \otimes \pi_2 \otimes \pi_3)P(\Phi)A =  e_1^{\otimes 3} + \hdots + e_{q+1}^{\otimes 3}$$

    \noindent where

    $$\pi_i: \CC^{3(q+1)} \to \CC^{q+1}, \, (a_1, a_2, a_3) \mapsto a_i.$$

    \end{itemize}
\end{ex}







    
    



We can easily construct complicated constructible subsets, for example like this:

\begin{ex} \label{example:artif}
    Let $P = \CC \times Q$, where $Q$ is any pure polynomial functor, $X^{(0)}, X^{(1)}, \hdots, X^{(n)}$ constructible subsets of $Q$. Then

    $$X = ((\CC \setminus \{1, \hdots, n\}) \times X^{(0)}) \cup (\{1\} \times X^{(1)}) \cup \hdots \cup (\{n\} \times X^{(n)})$$

    is a constructible subset.
\end{ex}

The following is an example of a pre-constructible subset that is not constructible:

\begin{ex}
    Let $P(V) = \CC \times V^{\otimes 2}$ and

    $$X(V) = ((\CC\setminus \ZZ_{\geq 0})\times V^{\otimes 2}) \cup \bigcup_{m\in\ZZ_{\geq 0}} \{m\} \times \{A \in V^{\otimes 2} \mid rk(A) \leq m\} $$

    Note that for every $n \in \NN$

    \begin{align*}
   X(\CC^n) =& ((\CC\setminus  \ZZ_{\geq 0})\times \CC^{n\times n}) \cup \bigcup_{m \geq n}  \{m\} \times \underbrace{\{A \in \CC^{n\times n} \mid rk(A) \leq m\}}_{=\CC^{n\times n}}  \cup \\
   & \bigcup_{m=0}^{n-1} \{m\} \times \{A \in \CC^{n\times n} \mid rk(A) \leq m\}\\
    =&((\CC\setminus \{0, \hdots, n-1\})\times \CC^{n\times n}) \cup \bigcup_{m=0}^{n-1} \{m\} \times \{A \in \CC^{n\times n} \mid rk(A) \leq m\} 
    \end{align*} 

    is constructible. But for every $n \in \ZZ_{\geq 0}$, the set  

    $$\{A \in P(V): \forall \phi \in \Hom(V, \CC^n), P(\phi)(A) \in X(\CC^n)\}$$

    is equal to 

    $$((\CC\setminus \{0, \hdots, n-1\})\times V^{\otimes 2}) \cup \bigcup_{m=0}^{n-1} \{m\} \times \{A \in V^{\otimes 2} \mid rk(A) \leq m\} $$

    which is not the same as $X(V)$ if $\dim(V) > n$.
\end{ex}

Finally, an example of a semialgebraic set, with no equivalent in the complex world:

\begin{ex} \label{example:psd}
    $P=S^2$ (i.e. symmetric matrices), and $X(V)$ are the positive semi-definite elements in $P(V)$. This is a semialgebraic subset determined by $\RR^1$, since for $A \in P(V) \setminus X(V)$, there exists $v \in V^{\ast}$ such that $vAv^{\top} = P(v)A < 0$, i.e. $P(v)A \notin X(\RR^1)$.
\end{ex}

For the following example, we do not know whether it is semialgebraic:

\begin{que}
    For $P=S^{2d}$, is the subset $X$ given by elements that can be written as sums of squares semialgebraic?
\end{que}

\subsection{Elementary Properties}

We will later need the following easy Proposition. Also here, the word constructible can be replaced by the word semialgebraic (which would implicitly change the field from $\CC$ to $\RR$).

\begin{prop}\label{proposition:properties}

If $X$ and $Y$ are constructible subsets of a polynomial functor $P$, and $\alpha:Q \to P$ is a polynomial transformation then

\begin{enumerate}[label=(\roman*)] \label{proposition:properties}
    \item \label{item:intersection} The intersection $(X \cap Y)(V):=X(V) \cap Y(V)$ is a constructible subset.
    \item \label{item:union} The union $(X \cup Y)(V):=X(V) \cup Y(V)$ is a constructible subset.
    \item \label{item:preimage} The preimage $\alpha^{-1}(P)(V):=\alpha^{-1}(P(V)) \subseteq Q(V)$ is a constructible subset.
\end{enumerate}

\end{prop}

\begin{proof}
Statements \ref{item:intersection} and \ref{item:preimage} are completely straightforward, so we will only prove \ref{item:union}: It is clear that $X \cup Y$ is pre-constructible. To prove that it is constructible, let $U_1$ and $U_2$ be the vector spaces that $X$ resp. $Y$ are determined by. We claim that $X \cup Y$ is determined by $U_1 \oplus U_2$.\\

Let $v \in P(V) \setminus (X \cup Y)(V)$. Then there exist $\phi_1: V \to U_1$, $\phi_2: V \to U_2$, such that $P(\phi_1)(v) \notin X(U_1)$ and $P(\phi_2)(v) \notin Y(U_2)$. Then $P(\phi_1 \oplus \phi_2)(v) \notin (X \cup Y)(U_1 \oplus U_2)$, because otherwise, denoting by $\pi_{U_1}$ and $\pi_{U_2}$ the corresponding projections from $U_1\oplus U_2$ onto $U_1$ and $U_2$, $P(\pi_{U_1})P(\phi_1 \oplus \phi_2)(v)=P(\phi_1)(v) \in X(U_1)$ and $P(\pi_{U_2})P(\phi_1 \oplus \phi_2)(v) = P(\phi_2)(v) \in Y(U_2)$. 
\end{proof}

\section{Parameterisation of Constructible Subsets} \label{section:parameterisation}

\subsection{Statement}

The goal of this section is to prove the following theorem that will be an important ingredient for our main Theorem \ref{Theorem:chevalley} but is also interesting in its own right:

\begin{thm}[Parameterisation of Constructible Subsets]
\label{Theorem:paramconst}
Let $P$ be a polynomial functor over $\CC$ and $X \subseteq P$ a constructible subset. Then there exist finitely many polynomial transformations

$$\alpha^{(i)}:A^{(i)} \times Q^{(i)} \to P$$

where $A^{(i)}$ are irreducible affine varieties, and $Q^{(i)}$ pure polynomial functors, such that 

$$X = \bigcup_i \im(\alpha^{(i)}).$$

\end{thm}

This theorem reduces the proof of our version of Chevalley's Theorem to showing that images of polynomial transformations on sets of the form $A\times Q$ as above are constructible. Note that we certainly need to allow $A$ to be an affine variety, and not a full affine space, because otherwise this theorem would imply that all constructible (and in particular all closed) sets, in the classical sense, are parameterizable by polynomials, which is well-known to be wrong.

\begin{re}
    The theorem is wrong for semialgebraic subsets. Let $P=S^2$ the symmetric matrices, and $X$ its positive semidefinite elements (as in Example \ref{example:psd}). Note that for every $V\in\Vec$, $X(V)$ has the same dimension as $S^2(V)$, namely $\binom{\dim(V)+1}{2}$, i.e. it is quadratic in $\dim(V)$. But if it was possible to cover $X$ by images of polynomial transformations, then by the classification of polynomial transformations, it would have to be covered by images of transformations of the form 

    $$\alpha^{(i)}:A^{(i)} \times T^{\oplus d} \to P.$$

    However, such a union of images can only have dimension linear in $\dim(V)$, which is a contradiction.

\end{re}

\subsection{Examples}

\begin{ex}
    $X$ as in Example \ref{example:span} is the image of the polynomial transformation
    \begin{align*}
    \CC^d \times V^{\oplus d} &\to V^{\oplus d+1} \\
    (a_1, \hdots, a_d, v_1, \hdots, v_d) &\mapsto (a_1v_1 + \hdots + a_dv_d, v_1, \hdots, v_d)
    \end{align*}
\end{ex}

\begin{ex} \label{example:quadforms}
    Let $P(V)=S^2(V) \oplus S^2(V)$ (where $S^2(V)$ is thought of as degree-2-homogeneous-polynomials), and

    $$X(V)=\{(f,g) \in P(V): \forall a \in V^{\ast},\,\,\, f(a)=0 \Rightarrow g(a)=0\}$$

    This is a constructible subset determined by $\CC^1$, because for $(f,g) \in P(V) \setminus X(V)$, there exists $a \in V^{\ast}$ such that $g(a) \neq 0$ and $f(a) = 0$, and hence $P(a)(f,g) \in P(\CC^1) \setminus X(\CC^1)$. It is also the union of the images of the following polynomial transformations:

    \begin{align*}
    \CC \times S^2 & \to S^2 \oplus S^2 &  S^1 \oplus S^1 & \to S^2 \oplus S^2\\
    (a, q) &\mapsto (q, a\cdot q) &  (l,m) &\mapsto (l^2, lm)
\end{align*} 
\end{ex}

\begin{ex}
    If $X$ is a closed subset (see Example \ref{example:closed}), then the following previously-known theorem says that it can be parameterized:
\end{ex}

\begin{thm} \label{Theorem:paramclosed}
    Let $P$ be a polynomial functor, and $X\subseteq P$ a closed subset that is not of the form $A \times P_{\geq 1}$ for some affine variety $A$. Then there exist finitely many polynomial transformations $\alpha^{(j)}: C_j \times Q_j \to P$ (with $C_j$ irreducible and closed, $Q_j < P_{\geq 1}$) such that ${X} = \bigcup_j \im(\alpha^{(j)})$.
\end{thm}

\begin{proof}
    See either Theorem 4.2.5. in \cite{Bik2020} or, for a more precise statement but in the language of GL-Varieties, Proposition 5.6. in \cite{Bik21}.
\end{proof}

In fact, the proof of Theorem \ref{Theorem:paramconst} relies heavily on this theorem.

\subsection{Proof}

The proof of Theorem \ref{Theorem:paramconst} needs one more result from \cite{Bik21b} (this is also the part that requires the ground field to be algebraically closed):

\begin{thm} \label{Theorem:dense}
    Let $P$ be a pure polynomial functor over $\CC$ and $U\in \Vec$. Then there exists $V\in\Vec$ and a dense open subset $\Sigma \subseteq P(V)$, such that for every $p\in \Sigma$ the map

       \begin{align*}
           \Hom(V, U) &\to P(U)\\
           \phi &\mapsto P(\phi)(p)
       \end{align*}

       is surjective.
\end{thm}

\begin{proof}
    The theorem follows directly from Corollary 2.5.4. in 
    \cite{Bik21b} with $V$ big enough, such that 
    $\Sigma:=P(V)\setminus (\bigcup_{i=1}^k \overline{\im \alpha_{i, V}})$ is dense (this is possible by a simple 
    dimensionality argument).
\end{proof}

\begin{proof}[Proof of Theorem \ref{Theorem:paramconst}] \,
    \begin{itemize}
        \item Write $\overline{X} = X^{(1)} \cup \hdots \cup X^{(n)}$ where $X^{(i)}$ are the closed irreducible components of $\overline{X}$. To prove that $X$ is parameterisable it suffices to prove that $X^{(i)} \cap X$ (which is again a constructible subset by \ref{proposition:properties}.\ref{item:intersection}) is parameterisable for every $i$. Hence, we can assume without loss of generality that $\overline{X}$ is irreducible.

        \item If $\overline{X}$ is not of the form $A \times P_{\geq 1}$ for some affine variety $A$, then by Theorem \ref{Theorem:paramclosed} there exist finitely many polynomial transformations $\beta^{(j)}: R_j \times Q_j \to P$ (with $R_j$ closed and irreducible, $Q_j < P_{\geq 1}$) such that $\overline{X} = \bigcup_j \im(\beta^{(j)})$.

        \item By Proposition \ref{proposition:properties}.\ref{item:preimage}, $(\beta^{(j)})^{-1}(X)$ are constructible subsets. By induction on the order of polynomial functors each of them can be covered by finitely many maps $\gamma^{(ji)}$, and hence $X$ is the union of the images of $\beta^{(j)} \circ \gamma^{(ji)}$.

        \item So assume that $\overline{X}=A\times P_{\geq 1}$, for some affine variety $A$. Note that $A$ is irreducible, since $\overline{X}$ is irreducible. Our next goal is to find a dense open subset $B \subseteq A$ such that $B \times P_{\geq 1} \subseteq X$.

        \item Consider the set

       $$\Omega:= \{b \in A : \{b\} \times P_{\geq 1}(U) \subseteq X(U)\}$$

       \noindent (where $U$ is the vector space that $X$ is determined by). $\Omega$ is constructible by quantifier elimination. We want to show that $\Omega$ is dense in $A$, so we can take $B$ as an appropriate subset of $\Omega$.

       \item By Theorem \ref{Theorem:dense} there exists a vector space $V$ and a dense open subset $\Sigma \subseteq P_{\geq 1}(V)$, such that for every $p\in \Sigma$ the map

       \begin{align*}
           \Hom(V, U) &\to P_{\geq 1}(U)\\
           \phi &\mapsto P_{\geq 1}(\phi)(p)
       \end{align*}

       \noindent is surjective. So in particular, if for some $p\in \Sigma$ and $b\in A$, $(b, p)$ lies in $X(V)$, then $b$ lies in $\Omega$.

       \item So $X(V) \subseteq (\Omega \times P_{\geq 1} (V)) \cup ((A\setminus \Omega)\times (P_{\geq 1}(V)\setminus \Sigma))$. But since we assumed that $\overline{X}=A\times P_{\geq 1}$ (so in particular $\overline{X(V)}=A\times P_{\geq 1}(V)$), $\Omega$ must be dense in $A$.

       \item Hence there exists a subset $B\subseteq \Omega$ that is open (and dense) in $A$, and so $B\times P_{\geq 1} \subseteq X$. Since $B$ is quasi-affine it can be written as a finite union of irreducible affine varieties, say $B_i$. $B \times P_{\geq 1}$ can be covered with the images of identity maps on $B_i \times P_{\geq 1}$, and $((A\setminus B) \times P_{\geq 1}) \cap X$ can be covered by induction using noetherianity of $A$.

    \end{itemize}
\end{proof}

\section{Chevalley's and Weak Tarski-Seidenberg's Theorems} \label{section:chevalley}

\subsection{Statement}

We now finally set out to prove our functorial version of Chevalley's Theorem, and a weaker version of Tarski-Seidenberg's Theorem:

\begin{thm} \label{Theorem:chevalley} \,
    \begin{enumerate}[label = (\roman*)]
    \item Let $P$, $Q$ be polynomial functors over $\CC$, $Y \subseteq Q$ a constructible subset, and $\alpha:Q \to P$ a polynomial transformation. Then, $X:=\alpha(Y)\subseteq P$ is a constructible subset.
    
    \item Let $P$, $Q$ be polynomial functors over $\RR$, $Y \subseteq Q$ a closed subset, and $\alpha:Q \to P$ a polynomial transformation. Then, $X:=\alpha(Y)\subseteq P$ is a semialgebraic subset.

    \end{enumerate}
\end{thm}

The statement reduces to the case $Y=A \times Q$ by Theorem \ref{Theorem:paramconst} for the complex case and by Theorem \ref{Theorem:paramclosed} for the real case. We conjecture that statement (i) remains true when taking $Y$ as semialgebraic, and not just closed, but our methods are insufficient to prove this.\\

We stress again, that statement (ii) is also true over any other field of characteristic 0, in the sense that $X$ is determined by a particular vector space $U$, since the proof of this part does not use any particular properties of $\RR$.

\subsection{Proof}

\begin{re}
     The proof of the second point of the theorem requires that Theorems \ref{Theorem:shift} and \ref{Theorem:paramclosed} hold not only over $\CC$, but also over $\RR$ (not just as schemes but as $\RR$-points). Even though the given sources do not explicitly state that this is the case, it is clear from the proofs that it is indeed the case.
\end{re}

The proof uses similar methods as the proof of Theorem \ref{Theorem:Noetherianity} and also consists of a double induction. We will need the following lemma as a sort of base case:

\begin{lem} \label{lemma:innerindstep}
    Let $\alpha: A \times Q \to P=P_0 \oplus \hdots \oplus P_d$ (with $Q$ a pure
    polynomial functor over $\RR$ or $\CC$, $A$ affine
    irreducible, $P_d$ not the zero-functor) a polynomial transformation,
    such that for $X:=\im(\alpha)$ we have that $\overline{X}$ is of the form $\widetilde{X} \times P_d$ ($\widetilde{X} \subseteq P_{\leq d-1}$). Then there exists an
    open dense subset $A'$ of $A$, such that $\alpha(A'\times Q)$
    is of the form $X' \times P_d$ (with $X' \subseteq P_{\leq d-1}$). 
\end{lem}

\begin{ex}
    Let $\alpha: A\times S^1 \oplus S^2 \to B\times S^1 \oplus S^2$ be a polynomial transformation. By \cite[Proposition 1.3.25]{Bik2020}, $\alpha$ is of the form

    \begin{align*}
    \alpha_V: (a, v, M) &\mapsto (f_1(a), f_2(a)v, f_3(a)v^2 + f_4(a)M) 
    \end{align*}

    (where $f_1:A \to B$ is a morphism, $f_2, f_3, f_4 \in K[A]$). Assume that $A$ is irreducible, and $\overline{\im(\alpha)}$ is of the form $\widetilde{X} \times S^2$. This implies that $f_4$ is not the zero-polynomial, since the degree-2-part of the image has to be of dimension quadratic in $\dim V$, and the image of $f_3(a)v^2$ only has dimension linear in $\dim V$. Set $A':=A\setminus\mathcal{V}(f_4)$. Then, 
    $$\alpha(A'\times S^1 \oplus S^2) = ((f(A' \cap \mathcal{V}(f_2)) \times \{0\}) \cup (f(A' \setminus \mathcal{V}(f_2)) \times S^1)) \times S^2$$

    which is of the required form.
\end{ex}

\begin{proof}[Proof of Lemma \ref{lemma:innerindstep}]\,
\begin{itemize}
    \item Let $\pi: P \to P_d$ be the standard projection (this is a linear, and in particular polynomial, transformation), and consider $\pi\circ\alpha$. By the conditions in the Lemma, this map is dominant.

    \item Write $Q=Q_{<d} \oplus Q_d \oplus Q_{>d}$, and accordingly write elements of $A \times Q(V)$ as $(a, q_{<d}, q_d, q_{>d})$. Then, by Remark \ref{remark: split2} we can write

    \[\pi_V\circ\alpha_V(a, q_{<d}, q_d, q_{>d}) = \alpha_{1, V}(a, q_{<d}) + \alpha_{2, V}(a, q_d)\]

    \noindent We claim that $\alpha_2$ has to be dominant: Indeed, the image of $\alpha_{1, V}$ has dimension of order $O(\dim(V)^{d-1})$, and if $\alpha_2$ were not dominant, its image would have codimension of order $O(\dim(V)^d)$.

    \item To further investigate what $\alpha_2$ looks like, write 

    $$Q_d = \bigoplus_{\lambda:|\lambda|=d} (S^{\lambda})^{\oplus m_{\lambda}},\,\,\, P_d = \bigoplus_{\lambda:|\lambda|=d} (S^{\lambda})^{\oplus n_{\lambda}}$$

    \noindent and

    $$\alpha_{2, V}(a, (q_{\lambda i})_{|\lambda|=d, 1 \leq i \leq m_{\lambda}}) = (p_{\lambda j})_{|\lambda|=d, 1 \leq j \leq n_{\lambda}}.$$

    \noindent So, $\alpha_2$ is given by polynomials $f_{\lambda i j} \in K[A]$ by

    \[p_{\lambda j} = \sum_{i=1}^{m_{\lambda}} f_{\lambda i j}(a)q_{\lambda i} \]

    \item Since $\alpha_2$ is dominant, for every partition $\lambda$, the matrix $(f_{\lambda i j}(a))_{ij}$ must be dominant (or, equivalently, surjective) for at least one $a \in A$. This implies that $m_{\lambda}\geq n_{\lambda}$, and that the variety

    \[B_{\lambda}:=\{a \in A : (f_{\lambda i j}(a))_{ij} \text{ has not full rank}\}\]

    \noindent is a proper closed subvariety of $A$.

    \item Set $A':= A \setminus (\bigcup_{\lambda:|\lambda|=d} B_{\lambda}).$ This is open by definition, and using irreducibility of $A$, we conclude that it is dense. We claim that $\alpha(A'\times Q)$
    is of the form $X' \times P_d$. Indeed, if $\alpha_V(a, q_{<d}, q_d, q_{>d}) = (p_{<d}, p_d)$ (with $a\in A'$) is in the image, then, since by construction $\alpha_{2,V}(a, \cdot)$ is surjective, we can modify $q_d$ to reach any other point of the form $(p_{<d}, p_d')$ with $p_d' \in P_d(V)$.
    
    \end{itemize}
\end{proof}

\begin{proof}[Proof of Theorem \ref{Theorem:chevalley}]

Recall that we want prove that for $\alpha:Q \to P$ and $Y\subseteq Q$ constructible/closed, $\alpha(Y)=:X$ is constructible/semialgebraic. As usual, the word constructible in the proof can be replaced by the word semialgebraic (and the name Chevalley by the names Tarski-Seidenberg) for a proof of the second point of the theorem, except for the steps where the two cases are explicitly treated differently.

\begin{itemize}

        \item It is clear that $X$ is a subset of $P$, and, by classical Chevalley's Theorem, that $X(V)$ is constructible for every $V$. So we just have to show that $X$ is determined by some vector space.
    
        \item By Theorem \ref{Theorem:paramconst} in the complex case, and Theorem $\ref{Theorem:paramclosed}$ in the real case, $Y$ can be written as
        \[Y=\bigcup_i \alpha^{(i)}(A^{(i)}\times R^{(i)})\]
        
        where $\alpha^{(i)}:A^{(i)}\times R^{(i)} \to Q$ are finitely many polynomial transformations, $A^{(i)}$ are irreducible affine varieties, and $R^{(i)}$ are pure polynomial functors. So $X$ is the union of the images of $\alpha \circ \alpha^{(i)}$, and since by Proposition \ref{proposition:properties}.\ref{item:union} finite unions of constructible subsets are again constructible subsets, it is enough to prove the theorem when $Y$ is of the form $Y = A \times Q$ (with $A$ affine-irreducible and $Q$ pure).

        \item The proof consists of a double induction: There is an outer induction hypothesis that assumes that all images of transformations $\alpha':A'\times Q' \to P'$ are constructible, where $A', Q'$ are arbitrary and $P'<P$. The inner induction hypothesis assumes that all images of maps $\alpha':A'\times Q' \to P$ are constructible, where either $Q'<Q$, or $Q' \cong Q$ and $A' \subsetneq A$ (but with fixed codomain $P$).

        \item If $X$ happens to be a subset of $P_0 \oplus \{(0, \hdots, 0)\}$, then it is constructible by classical Chevalley's Theorem.

        \item If $\overline{X}$ is of the form $\widetilde{X} \times P_d$ as in the previous Lemma \ref{lemma:innerindstep}, then we can use the lemma to conclude that there exists an open, dense $A' \subseteq A$, such that $\alpha(A' \times Q)$ is of the form $X' \times P_d$. Now, $X' \times P_d$ is constructible if and only if $X'$ is, but by Remark \ref{remark: split2}, $X'$ can be identified with the image of $\alpha$ restricted to $A' \times Q_{\leq d-1}$, so it is constructible (using either the inner or the outer induction hypothesis). And $\alpha((A \setminus A') \times Q)$ is constructible by the inner induction hypothesis, hence $X$ is constructible as the union of two constructible subsets.

        \item If $\overline{X}$ is of neither of the two above forms, then by the Shift Theorem (Theorem \ref{Theorem:shift}), there exist a vector space $U$, and a nonzero polynomial $h \in K[\overline{X(U)}]$, such that $\Sh_U(\overline{X})[1/h]$ is isomorphic to some $B\times P'$ via an isomorphism $\beta$, where $B$ is an affine variety, and $P'$ is a pure polynomial functor with $P' < P_{\geq 1}$.

        \item We take a step back and discuss the strategy for the rest of the proof: Let $p\in P(V)\setminus X(V)$. We need to show that there exists a vector space $W$, independent of $p$ and $V$, such that there exists a linear map $\phi:V \to W$ with $P(\phi)(p) \notin X(W)$. If $p \notin \overline{X}(V)$ then we already know from Example \ref{example:closed} that there does exist such a vector space, say $W'$. If $p \in \overline{X}(V)$, then we consider two cases, namely

        \[p \in Z_1(V):= \{p \in \overline{X}(V)|\text{for all }\phi:V\to U,\, h(P(\phi)p)=0\}\]

        \noindent and

        \[p \in Z_2(V):= \overline{X}(V)\setminus Z_1(V).\]

        \noindent The first case can be dealt with by the inner induction (again using unirationality), and for the second case we will use the Shift Theorem \ref{Theorem:shift} and the outer induction, even though this will need a little more care, as $Z_2$ is typically not even a subset of $P$.

        \item We quickly do the first case: Note that $X \cap Z_1$ is the image of $\alpha$ restricted to $Y':=\alpha^{-1}(Z_1)$ which is a closed proper subset of $Y=A \times Q$. Then either $Y'=A'\times Q$ where $A' \subsetneq A$ and we can use the induction hypothesis directly to see that $X \cap Z_1 = \alpha(A'\times Q)$ is constructible. Or, $Y'$ is not of this form, but then by Theorem \ref{Theorem:paramclosed}, $Y'$ is the union of finitely many images of maps $\alpha'^{(i)}:A'^{(i)}\times R'^{(i)} \to A \times Q$ with $R'^{(i)}<Q$, so $X \cap Z_1$, which is the union of all images $\alpha \circ \alpha'^{(i)}$, is constructible by the inner induction hypothesis. So there exists $W_1 \in \Vec$, such that for all $p\in Z_1(V)\setminus X(V)$, there exists $\phi:V\to W_1$ with $P(\phi)p \notin X(W_1)$.


        \item For the second case, consider first
        
        \[Z_2'(V):=\{p \in \overline{X(U \oplus 
        V)}|h(P(\pi_V)p)\neq 0\}  \subseteq \Sh_U(P)(V)\] 
        
        \noindent (where $\pi_V:U\oplus V \to U$ is the standard projection). Note that $h$ only contains variables from the degree-0-part of $\Sh_U P$. So 
        $Y'(V):=\alpha_{U \oplus V}^{-1}(Z_2'(V)) \subseteq \Sh_U Q(V)$ is of the form $A'\times Q''$, where $Q''$ is the pure part of $\Sh_U Q$ and $A'$ is an affine subvariety of $A \times Q(U)$. So we can use the outer induction hypothesis to conclude that $\beta_V \circ \alpha_{U\oplus V}(Y'(V)) = \beta_V(Z_2'(V) \cap X(U \oplus V))$ is a 
        constructible subset of $B \times P'$ (since $P' < P_{\geq 1}$).

        \item So we get that there exists a vector space $W_2$, such that for all $p \in Z_2'(V)\setminus X(U \oplus V)$, there exists a linear map $\phi: V \to W_2$ such that $P(\id \oplus \phi)p \in Z_2'(W_2)\setminus X(U \oplus W_2)$.

        \item Now, finally, let $p\in Z_2(V) \setminus X(V)$. We first assume that $\dim(V)\geq\dim(U)$, so $V$ is isomorphic to a vector space of the form $U\oplus V'$, and we can think of $p$ as an element in $Z_2(U\oplus V')\setminus X(U\oplus V')$. By definition of $Z_2$ there exists $\psi:U \oplus V' \to U$ such that $h(P(\psi)p) \neq 0$. This is an open condition on $\psi$, so we can also assume that $\psi$ has full rank. Hence, there exists $g \in \GL(U \oplus V')$ such that $\psi = \pi_V \circ g$, and therefore $P(g)p \in Z_2'(V') \setminus X(U\oplus V')$. Using now the map $\phi:V'\to W_2$ from the previous bullet point we get

        
        \[P((\id_U \oplus \phi) \circ g)p \in Z_2'(W_2)\setminus X(U \oplus W_2) \subseteq P(U \oplus W_2)\setminus X(U \oplus W_2).\]

        \item So, if $\dim(V)\geq\dim(U)$ we are done, and if $\dim(V)<\dim(U)$ we can instead of $(\id_U \oplus \phi) \circ g$ simply use an inclusion map $\iota: V \to U \oplus W_2$ such that $P(\iota)p \in P(U \oplus W_2)\setminus X(U \oplus W_2).$

        \item Hence, $X$ is determined by $K^{\min(\dim(W'), \dim(W_1), \dim(U\oplus W_2))}$. \qedhere
        
    \end{itemize}

\end{proof}

\bibliographystyle{alphaurl}
\bibliography{draismajournal,draismapreprint,diffeq, constructible}

\end{document}